\documentclass[12pt]{amsart}
\usepackage[top=30truemm,bottom=30truemm,left=25truemm,right=25truemm]{geometry}
\usepackage{txfonts}
\usepackage{mathrsfs}

\usepackage{color}
\usepackage{bm}
\usepackage{amsfonts,amssymb}
\usepackage{dsfont}
\usepackage{amscd}
\usepackage{extarrows}
\usepackage{amsmath}
\usepackage{mathrsfs}
\usepackage{enumerate}
\usepackage{amscd}
\usepackage[all]{xy}
\usepackage[pagebackref,colorlinks]{hyperref}

\usepackage{extarrows}

\makeatletter
\@namedef{subjclassname@2020}{%
  \textup{2020} Mathematics Subject Classification}
\makeatother

\theoremstyle{plain} 
\newtheorem{theorem}{\indent\bf Theorem}[section]

\newtheorem{conj}[theorem]{\indent\bf Conjecture}
\theoremstyle{definition} 

\newtheorem{thm}{Theorem}[section]

\newtheorem{lem}[thm]{Lemma}

\theoremstyle{definition}

\theoremstyle{remark}
\newtheorem{rem}{Remark}[section]

\begin{document}

\title[Demailly's conjecture]{A note on Demailly's transcendental Morse inequalities conjecture}

	\author[Y. Li]{Yinji Li}
\address{Yinji Li:  Institute of Mathematics\\Academy of Mathematics and Systems Sciences\\Chinese Academy of
	Sciences\\Beijing\\100190\\P. R. China}
\email{1141287853@qq.com}
\author[Z. Wang]{Zhiwei Wang}
\address{Zhiwei Wang: Laboratory of Mathematics and Complex Systems (Ministry of Education)\\ School of Mathematical Sciences\\ Beijing Normal University\\ Beijing 100875\\ P. R. China}
\email{zhiwei@bnu.edu.cn}
\author[X. Zhou]{Xiangyu Zhou}
\address{Xiangyu Zhou: Institute of Mathematics\\Academy of Mathematics and Systems Sciences\\and Hua Loo-Keng Key
	Laboratory of Mathematics\\Chinese Academy of
	Sciences\\Beijing\\100190\\P. R. China}
\address{School of
	Mathematical Sciences, University of Chinese Academy of Sciences,
	Beijing 100049, P. R. China}
\email{xyzhou@math.ac.cn}

	\begin{abstract} 	
Let $(X,\omega)$ be an $n$-dimensional  compact Hermitian manifold with $\omega$ a pluriclosed Hermitian metric, i.e. $dd^c\omega=0$. Let $\{\alpha\},\{\beta \}\in H^{1,1}_{BC}(X,\mathbb R)$ be two nef classes, such that $\alpha^n-n\alpha^{n-1}\cdot\beta>0$. In this short note, we prove that if there is a bounded quasi-plurisubharmonic  potential $\rho$, such that $\alpha+dd^c\rho\geq 0$ in the weak sense of currents, then the class $\{\alpha-\beta\}$ contains a K\"ahler current. This gives a partial solution of Demailly's transcendental Morse inequalities conjecture.
	\end{abstract}
	
%
%
	
	\subjclass[2020]{}


	\dedicatory{}

\maketitle

\section{Introduction}

The main concern of this paper is the following Demailly's transcendental Morse inequalities conjecture.

\begin{conj}[see {\cite[Conjecture 10.1 (ii)]{BDPP13}}]\label{conj}
Let $X$ be a compact complex manifold, and $n=\dim X$.
Let $\{\alpha\}$ and $\{\beta \}$ be nef cohomology classes of type $(1,1)$ on $X$ satisfying the inequality $\alpha^n-n\alpha^{n-1}\cdot \beta>0$. Then $\{\alpha-\beta \}$ contains a K\"ahler current and $$Vol(\alpha-\beta)\geq \alpha^n-n\alpha^{n-1}\cdot\beta.$$
	\end{conj}
This conjecture is known to imply many other important conjectures, including the characterization of the movable cone as the dual of the pseudoeffective cone,
 the $\mathcal C^1$-differentiability of the volume function in the cone of big classes, and the orthogonality of divisorial Zariski decompositions. 
 
 When $X$ is assumed to be K\"ahler,  the class $\{\alpha\}$ is integral and $\{\beta \}=0$,   a stronger version (see \cite[Conjecture 10.1 (i)]{BDPP13}) was proved by Boucksom \cite{Bou02}.
 
 When $X$ is projective, Conjecture \ref{conj} was recently settled by D. Witt Nystr\"om \cite{Nys19}. 
 
 When there is a Hermitian metric $\omega$ with $dd^c\omega^l=0$ for all $l$ (especially when $X$ is K\"ahler), inspired by Chiose \cite{Ch16}, Xiao \cite{Xiao15} proved a weak version of  the   Conjecture \ref{conj}, say  $\{\alpha-\beta \}$ contains a K\"ahler current, if $\alpha^n-4n\alpha^{n-1}\cdot \beta>0$. Shortly later, following Xiao's method with a clever estimate in the Monge-Amp\`ere equation,  the Conjecture \ref{conj}  was proved by Popovici \cite{Pop16} under the same hypothesis as Xiao on the Hermitian metric $\omega$. For more related work on volume and self-intersection of diffenreces of two nef classes, see \cite{Pop17}.
 
 More recently, Guedj-Lu \cite{GL22} introduced an interesting quantity $v_+(\omega)$ for  a Hermitian metric $\omega$ on $X$ as 
 \begin{align*}
 	v_+(\omega):=\sup\left\{ \int_X(\omega+dd^c\varphi)^n:  \varphi\in \mbox{PSH}(X,\omega)\cap L^\infty(X)  \right\}.
 	\end{align*}
 They proved in \cite{GL22} that on a compact Hermitian manifold $(X,\omega)$, if $v_+(\omega)<+\infty$, then the Conjecture \ref{conj} holds true.

 In this short note, based on our previous work on degenerate complex Monge-Amp\`ere equations \cite{LWZ23}, we can relax the assumption on the Hermitian metric $\omega$, by establishing the following
 
\begin{thm}\label{thm:main}
	Let $(X,\omega)$ be a compact Hermitian manifold, with $\omega$ a pluriclosed Hermitian metric, i.e. $dd^c\omega=0$. Let $\{\alpha\},\{\beta\}\in H^{(1,1)}(X,\mathbb{R})$ be real (1,1)-classes such that $\alpha$ has  a bounded quasi-plurisubharmonic  potential $\rho$, such that $\alpha+dd^c\rho\geq 0$ in the weak sense of currents, and $\{\beta\}$ is nef.  If    $\alpha^n\textgreater n\alpha^{n-1}\cdot\beta$, then $\{\alpha-\beta\}$ contains a K\"ahler current.
\end{thm}

\begin{rem}In \cite{GL22}, it was pointed out that  when $\dim X\leq 2$ or $\dim X=3$ and $dd^c\omega=0$,  $v_+(\omega)<+\infty$. Thus our Theorem \ref{thm:main} is new in $\dim X\geq 4$.
	\end{rem}

A pluri-closed Hermitian metric is also called Strong KT metric in many literature, we refer to a nice survey paper by Fino and Tomassini \cite{FT09}.

The method of  proof is inspired by Popovici \cite{Pop16}.  It is  based on  Lamari's lemma,  solutions to degenerate complex Monge-Amp\`ere equations \cite[Theorem 1.1]{LWZ23}, and   a generalization of Popovici's estimates in the Monge-Amp\`ere equation to the singular setting \cite[Lemma 6.5]{LWZ23}.

\subsection*{Acknowledgements} 
This research is supported by National Key R\&D Program of China (No. 2021YFA1002600, No. 2021YFA1003100).  The authors are partially supported respectively by NSFC grants (12071035, 11688101).

\section{Proof of Theorem \ref{thm:main}}
In this section, we give the proof of Theorem \ref{thm:main}. The following lemma  is due to Lamari.
\begin{lem}[\cite{Lam1}]\label{lem: lamari}
	Let $X$ be a compact Hermitian manifold, and let $\alpha$ be a smooth real (1,1)-form, there exists a distribution $\psi$ on X such that $\alpha+dd^c\psi\geq0$ if and only if:
	
	$$\int_X\alpha \wedge \gamma^{n-1}\geq0$$
	for any Gauduchon metric $\gamma$ on X.
\end{lem}
The following solution of degenerate complex Monge-Amp\`ere equations is a key ingredient in the proof.

\begin{thm}[{\cite[Theorem 1.1]{LWZ23}}]\label{thm:CMA}Let $(X,\omega)$ be a compact Hermitian manifold  of complex dimension $n$.  Let $\beta$ be a smooth real $(1,1)$-form on $X$ such that there exists  $\rho\in \mbox{PSH}(X,\beta)\cap L^\infty(X)$. Let $0\leq f\in L^p(X,\omega^n)$, $p>1$, be such that $\int_Xf\omega^n=\int_X\beta^n>0$. Then there exists a unique  real-valued function $\varphi\in \mbox{PSH}(X,\beta )\cap L^\infty(X)$, satisfying
	\[(\beta+dd^c\varphi)^n=f\omega^n\]
	in the weak sense of currents, and $\|\varphi\|_{L^\infty(X)}\leq C(X,\omega,\beta, M,\|f\|_p)$.
\end{thm}

Combining Theorem \ref{thm:CMA} and Bedford-Taylor theory \cite{BT82}, the authors get a generalization of  Popovici's estimates in the Monge-Amp\`ere equation to the singular setting.
\begin{lem}[{\cite[Lemma 6.5]{LWZ23}}]\label{lem:LWZ}
Let $X$ be an $n$-dimensional compact Hermitian manifold with a Gauduchon metric $g$. Set $G=g^{n-1}$.	Let $\alpha$ be a smooth closed  $(1,1)$-form such that there exists a bounded quasi-plurisubharmonic function $\rho$ such that $\alpha+dd^c\rho\geq 0$ in the weak sense of currents. Let $\omega$ be a Hermitian metric on $X$. Then for any bounded quasi-plurisubharmonic function $v$, and $c>0$, solving the following complex Monge-Amp\`ere equation 
$$(\alpha+dd^c\rho+dd^cv)^n=c\omega\wedge G,$$
it holds that 
\begin{align*}
	& \left(\int_X(\alpha+dd^c\rho+dd^cv)\wedge G\right)\cdot\left(\int_X(\alpha+dd^c\rho+dd^cv)^{n-1}\wedge \omega\right) \\
&\geq\frac{c}{n}\left(\int_X\omega\wedge G\right)^2.
\end{align*}
	\end{lem}

%

\begin{proof}[Proof of Theorem \ref{thm:main}]
	
By Lamari's lemma (Lemma \ref{lem: lamari}), 	 the class  $\{\alpha-\beta\}$  contains a  K\"ahler current if and only if there exists a $\delta>0$, such that
	\begin{align*}
\int_X\alpha^n=\int_X(\alpha+dd^c\rho)\wedge g^{n-1}\geq\int_X(\beta+\delta\omega)\wedge g^{n-1},
\end{align*}
for every Gauduchon metric $g$ on $X$.	

Now we prove Theorem \ref{thm:main} by contradiction.

Suppose to the contrary  that $\{\alpha-\beta\}$ does not contain any K\"ahler current,  there is a sequence of positive numbers $\{\delta_j\}$ decreases to $0$, and for any $j\in\mathbb{N^*}$,  there is a Gauduchon metric $g_j$ such that
	\begin{align*}
	\int_X(\alpha+dd^c\rho)\wedge g_j^{n-1}\leq \int_X(\beta+\delta_j\omega)\wedge g_j^{n-1}.
	\end{align*}
Due to the nefness of $\{\beta\}$, 	 for every $j$, there exists $\tilde{\rho_j}\in C^{\infty}(X)$ such that $\beta+dd^c\tilde{\rho_j}+\delta_j\omega$ is a Hermitian form on $X$. Set $G_j:=g_j^{n-1}$, by Theorem \ref{thm:CMA}, we can solve the following complex Monge-Amp\`ere equations:
	\begin{align*}
	(\alpha+dd^c\rho+dd^cv_j)^n=c_j(\beta+dd^c\tilde{\rho_j}+\delta_j\omega)\wedge G_j, \ v_j\in\mbox{PSH}(X,\alpha+dd^c\rho)\cap L^{\infty}(X).
	\end{align*}
	It is obvious that
	\begin{align*}
	c_j=\frac{\int_X\alpha^n}{\int_X(\beta+\delta_j\omega)\wedge G_j}.
	\end{align*}
	By Lemma \ref{lem:LWZ},  we get that
	\begin{align*}
	& \left(\int_X(\alpha+dd^c\rho+dd^cv_j)\wedge G_j\right)\cdot\left(\int_X(\alpha+dd^c\rho+dd^cv_j)^{n-1}\wedge (\beta+dd^c\tilde{\rho_j}+\delta_j\omega)\right) \\
	&\geq\frac{c_j}{n}\left(\int_X(\beta+dd^c\tilde{\rho_j}+\delta_j\omega)\wedge G_j\right)^2.
	\end{align*}
	This gives that
	\begin{align*}
	n\int_X(\alpha+dd^c\rho+dd^cv_j)^{n-1}\wedge (\beta+dd^c\tilde{\rho_j}+\delta_j\omega)\geq \int_X\alpha^n.
	\end{align*}
	However, $dd^c\omega=0$ leads to $\int_X(\alpha+dd^c\rho+dd^cv_j)^{n-1}\wedge (\beta+dd^c\tilde{\rho_j}+\delta_j\omega)=\int_X\alpha^{n-1}\wedge (\beta+\delta_j\omega)$. Letting  $j\rightarrow\infty$, we get that 
	\begin{align*}
\int_X\alpha^n\leq 	n\int_Xa^{n-1}\wedge\beta,
	\end{align*}
	which contradicts with the assumption that 
	$$\alpha^n- n\alpha^{n-1}\cdot\beta\textgreater0.$$
	The proof is complete.
\end{proof}




\end{document}